\documentclass{amsart}

\usepackage{dsfont}            
\usepackage{amssymb,amsmath,amsfonts}
\usepackage{mathrsfs}
\usepackage{a4wide}
\usepackage[all]{xy}
\usepackage[]{graphicx}        

\newcommand{\cst}{\ifmmode\mathrm{C}^*\else{$\mathrm{C}^*$}\fi}
\newcommand{\st}{\;\vline\;}
\newcommand{\CC}{\mathbb{C}}
\newcommand{\RR}{\mathbb{R}}
\newcommand{\NN}{\mathbb{N}}
\newcommand{\TT}{\mathbb{T}}
\newcommand{\KK}{\mathbb{K}}

\newcommand{\QG}{\mathbb{G}}
\newcommand{\tens}{\otimes}
\newcommand{\vtens}{\,\bar{\otimes}\,}
\newcommand{\id}{\mathrm{id}}
\newcommand{\comp}{\circ}
\newcommand{\uu}{{\scriptscriptstyle\mathrm{u}}}
\newcommand{\I}{\mathds{1}}
\newcommand{\cA}{\mathscr{A}}
\newcommand{\QH}{\mathbb{H}}
\newcommand{\sA}{\mathsf{A}}
\newcommand{\sB}{\mathsf{B}}
\newcommand{\sM}{\mathsf{M}}
\newcommand{\sZ}{\mathsf{Z}}
\newcommand{\sN}{\mathsf{N}}
\newcommand{\sX}{\mathsf{X}}
\newcommand{\sY}{\mathsf{Y}}
\newcommand{\sL}{\mathsf{L}}
\newcommand{\hh}[1]{\widehat{#1}}
\newcommand{\dd}[1]{\widetilde{#1}}
\newcommand{\op}{\text{\rm\tiny{op}}}
\newcommand{\flip}{\boldsymbol{\sigma}}
\newcommand{\ww}{\mathrm{W}}
\newcommand{\WW}{{\mathds{V}\!\!\text{\reflectbox{$\mathds{V}$}}}}
\newcommand{\Ww}{\mathds{W}}
\newcommand{\wW}{\text{\reflectbox{$\Ww$}}\:\!} 
\newcommand{\alR}{\boldsymbol{\alpha}_{\text{\tiny\rm{R}}}}
\newcommand{\alL}{\boldsymbol{\alpha}_{\text{\tiny\rm{L}}}}
\newcommand{\LL}{\mathbb{L}}
\newcommand{\tp}{\;\!\xymatrix{*+<.1ex>[o][F-]{\raisebox{-1.4ex}{$\scriptstyle\top$}}}\;\!}

\DeclareMathOperator{\Pol}{\mathcal{O}}
\DeclareMathOperator{\C}{C}
\DeclareMathOperator{\linf}{\ell^\infty\;\!\!}
\DeclareMathOperator{\cZ}{\mathscr{Z}}
\DeclareMathOperator{\B}{B}
\DeclareMathOperator{\vN}{vN}
\DeclareMathOperator{\Mor}{Mor}
\DeclareMathOperator{\M}{M}
\DeclareMathOperator{\Inn}{Inn}
\DeclareMathOperator{\linW}{\overline{span}^{\:\!\text{\tiny\rm{w}}}\:\!\!}
\DeclareMathOperator{\Linf}{\mathnormal{L}^\infty\;\!\!}
\DeclareMathOperator{\Ltwo}{\mathnormal{L}^2\;\!\!}
\DeclareMathOperator{\Irr}{Irr}
\DeclareMathOperator{\c0}{c_0}

\newtheorem{proposition}{Proposition}[section]
  \newtheorem{theorem}[proposition]{Theorem}
  \newtheorem{corollary}[proposition]{Corollary}
  \newtheorem{lemma}[proposition]{Lemma}
\theoremstyle{definition}
  \newtheorem{definition}[proposition]{Definition}
  \newtheorem{remark}[proposition]{Remark}
  \newtheorem*{q}{Question}

\numberwithin{equation}{section}

\begin{document}






\title{The canonical central exact sequence for locally compact quantum groups}

\author{Pawe{\l} Kasprzak}
\address{Department of Mathematical Methods in Physics, Faculty of Physics, University of Warsaw}

\author{Adam Skalski}
\address{Institute of Mathematics of the Polish Academy of Sciences}

\author{Piotr Miko{\l}aj So{\l}tan}

\begin{abstract}
For a locally compact quantum group $\mathbb{G}$ we define its center, $\mathscr{Z}(\mathbb{G})$, and its quantum group of inner automorphisms, $\mathrm{Inn}(\mathbb{G})$. We show that one obtains a natural isomorphism between $\mathrm{Inn}(\mathbb{G})$ and $\mathbb{G}/\!\mathscr{Z}(\mathbb{G})$, we characterize normal quantum subgroups of a compact quantum group as those left invariant by the action of the quantum group of inner automorphisms and discuss several examples.
\end{abstract}

\maketitle

\section{Introduction}\label{intro}

If $G$ is a classical group, it is elementary that its center, $\cZ(G)$, and its group of inner automorphisms, $\Inn(G)$ fit into a short exact sequence of groups
\[
\xymatrix{\{e\}\ar[r]&\mathscr{Z}(G)\ar[r]&G\ar[r]&\Inn(G)\ar[r]&\{e\}},
\]
where $\{e\}$ denotes the one element group. If $G$ is in addition equipped with a topology making it a locally compact group, $\cZ(G)$ is its closed (normal) subgroup, and $\Inn(G)$ admits a natural locally compact topology (of uniform convergence on compact subsets of $G$) and the sequence above becomes a short exact sequence in the category of locally compact groups.

Recent twenty years brought a large increase of interest in studying extensions of various group-theoretic questions to the context of compact quantum groups of Woronowicz \cite{pseudogr} and later to the category of locally compact quantum groups in the sense of Kustermans and Vaes \cite{KV}. Following a general approach of non-commutative mathematics, this requires reinterpreting various classical notions and properties in terms of appropriate function algebras. As a rule, locally compact quantum group extensions become usually significantly more intricate than their compact/discrete counterparts. Thus, for example, the notion of a closed quantum subgroup of a compact quantum group is relatively straightforward, and dates back already to \cite{podles}, and that of a locally compact quantum subgroup appeared in literature only in \cite{VaesVainerman} and was later investigated in detail in \cite{DKSS}. One significant difference lies in the fact that compact quantum groups are usually studied via their \cst-algebras (or even Hopf $*$-algebras), whereas the general locally compact situation is often best understood via von Neumann algebraic objects -- this will also be the approach taken in this paper.

Compact quantum groups are amenable to various universal constructions. Thus it should not be a surprise that a notion of the center of a compact quantum group $\QG$ introduced in \cite{Patri} (see also \cite{Chirva} and earlier \cite{wangsimple}), inspired by the work of M\"uger \cite{Mug} in the classical case, is based on one hand on the universal property ($\cZ(\QG)$ is the largest \emph{central} closed subgroup of $\QG$) and on the other on the categorical approach to compact quantum groups via their representation theory. One should note that central subgroups, so in particular $\cZ(\QG)$, are automatically \emph{cocommutative}, or in other words \emph{abelian} (this terminology should not be confused with the statement that say $\C(\QG)$ is abelian -- we then say that $\QG$ is \emph{classical}). The description of $\cZ(\QG)$, although in a sense non-explicit, allowed the authors of \cite{Patri} and \cite{Chirva} to determine the center in several concrete examples. In particular one can easily see that $\cZ(\QG)$ is a \emph{normal} subgroup of $\QG$ in the sense of \cite{wang3}, and thus one can also consider the quotient compact quantum group $\QG/\!\cZ(\QG)$, which in the case of $\QG=G$ being classical would be isomorphic to the group $\Inn(G)$.

The starting point for our considerations in this paper is an observation that a beautiful theorem of Baaj and Vaes from \cite{BV}, characterizing closed quantum subgroups of a locally compact quantum group $\QG$ via what we call \emph{Baaj-Vaes subalgebras} of $\Linf(\hh{\QG})$ ($\hh{\QG}$ denotes the locally compact quantum group dual to $\QG$), provides a very useful replacement for the universal \cst-algebra constructions familiar from the compact setting, and in particular permits a natural definition of $\cZ(\QG)$, again as a certain abstract maximal object. The center $\cZ(\QG)$ is a normal (in the sense of \cite{VaesVainerman2} or \cite{ext}) subgroup of $\QG$, and so we can speak about the quotient locally compact quantum subgroup $\QG/\!\cZ(\QG)$. We show that one can construct an explicit von Neumann algebra $\sM$ inside $\Linf(\QG)$ which is a Baaj-Vaes subalgebra, and so carries a natural structure of an algebra of essentially bounded functions on a locally compact quantum group, so that in the classical case the resulting locally compact group is precisely the group of inner automorphisms of the initial group. By analogy we denote $\sM$ by $\Linf\bigl(\Inn(\QG)\bigr)$ and call the locally compact quantum group $\Inn(\QG)$ the \emph{quantum group of inner automorphisms of} $\QG$. Thus we obtain the following exact sequence of locally compact quantum groups:
\[
\xymatrix{\{e\}\ar[r]&\mathscr{Z}(\QG)\ar[r]&\QG\ar[r]&\Inn(\QG)\ar[r]&\{e\}}.
\]

We show that in general $\Inn(\QG)$ has a natural action on $\hh{\QG}$, replacing in a sense the natural action of $\Inn(G)$ on $G$ by inner automorphisms, and that invariance under this action can be used to characterize normality of a subgroup of a quantum group as defined by Wang in \cite{free}. We present several examples of the studied constructions, in particular discussing the case of cocycle twists and duals of Drinfeld-Jimbo deformations.

Let us mention here that on the Hopf algebraic level, the Hopf center $\mathcal{ZH}$ of a given Hopf algebra $\mathcal{H}$ was defined in \cite{AD1} as the largest Hopf subalgebra of $\mathcal{H}$ contained in its center. The dual notion, that of the Hopf cocenter $\mathcal{CH}$, was also introduced therein. In \cite{CHK} the corresponding pair of exact sequences of Hopf algebras was shown to exist  and the issue of faithful (co)flatness was addressed.

The detailed plan of the paper is as follows: in the next subsection we recall some of the basic facts and terminology concerning locally compact quantum groups. Section \ref{centre} introduces the definition of the center of a locally compact quantum group $\QG$ and its basic properties, and Section \ref{Inner} does the same for the quantum group of inner automorphisms of $\QG$; these two sections explain also the short exact sequence from the title of the paper. Further in Section \ref{InnAct} we investigate the action of $\Inn(\QG)$ on the dual quantum group $\hh{\QG}$ and relate it to the notion of normality for closed quantum subgroups and in a short Section \ref{examples} we discuss the cases of cocycle twists and duals of Drinfeld-Jimbo deformations.

\subsection{Background and terminology}

By a morphism between von Neumann algebras we will mean a unital, normal $*$-homomorphism, and by a morphism between \cst-algebras $\sA$ and $\sB$ a non-degenerate $*$-homomorphism from $\sA$ to $\M(\sB)$. The symbol $\vtens$ denotes the von Neumann algebraic tensor product, $\otimes$ the minimal/spatial tensor product of $\cst$-algebras. The symbol $\flip$ will always denote the tensor flip map.

Throughout the article symbols $\QG$ and $\QH$ will denote \emph{locally compact quantum groups in the sense of Kustermans and Vaes} \cite{KV}. Here $\QG$ is a virtual object, primarily studied studied via its associated von Neumann algebra of `bounded measurable functions', $\Linf(\QG)$, but also via the \cst-algebra $\C_0(\QG)$ or its universal counterpart $\C_0^\uu(\QG)$ (the canonical \emph{reducing quotient map} from $\C_0^\uu(\QG)$ onto $\C_0(\QG)$ will be denoted $\Lambda_{\QG}$). In the special case of $\QG$ being \emph{discrete} (see \cite[Section 3]{PodSLW}), we will write $\linf(\QG)$ instead of $\Linf(\QG)$.

The dual locally compact quantum group of $\QG$ will be denoted by $\hh{\QG}$; recall that both $\Linf(\QG)$ and $\Linf(\hh{\QG})$ are represented on the same Hilbert space $\Ltwo(\QG)$. We will generally follow the terminology and conventions of the articles \cite{DKSS,KaspSol,proj,ext}; in particular the multiplicative unitary $\ww^{\QG}$ is an element of $\Linf(\hh{\QG})\vtens\Linf(\QG)$ and the \emph{comultiplication} of $\Linf(\QG)$ is given by the formula $\Delta_\QG(x)=\ww^{\QG}(x\tens\I){\ww^{\QG}}^*$, $x\in\Linf(\QG)$. We will very often use the fact that the \emph{left slices} of $\ww^{\QG}$ by normal functionals generate $\Linf(\QG)$ and the corresponding \emph{right slices} of $\ww^{\QG}$ generate $\Linf(\hh{\QG})$. We say that $\QG$ is \emph{classical} if $\Linf(\QG)$ is commutative, as then indeed $\Linf(\QG)=\Linf(G)$ for a uniquely determined locally compact group $G$, and that $\QG$ is \emph{abelian} (also called \emph{cocommutative} by some authors) if the coproduct $\Delta_{\QG}$ is symmetric: $\Delta_{\QG}=\flip\comp\Delta_{\QG}$. In the latter case $\Linf(\QG)$ is the group von Neumann algebra $\vN(G)$ for some locally compact group $G$. Let us also mention here that $\QG^{\op}$ denotes the \emph{opposite locally compact quantum group} of $\QG$, which arises from the same von Neumann algebra $\Linf(\QG)$ equipped with the flipped coproduct $\Delta_{\QG^{\op}}=\flip\comp\Delta_{\QG}$. Symbols $R$, $J$, $\tau$ (often adorned with the superscript $\QG$) will denote respectively the \emph{unitary antipode}, \emph{modular conjugation} and the \emph{scaling group} of $\QG$. We will also use the leg numbering notation, popular in the quantum group literature, without further comment.

Let $\QG$, $\QH$ be locally compact quantum groups. A \emph{bicharacter from $\QG$} to $\QH$ is a unitary $V\in\Linf(\hh{\QH})\vtens\Linf(\QG)$ such that $(\Delta_{\hh{\QH}}\tens\id)V=V_{23}V_{13}$ and $(\id\tens\Delta_{\QG})V=V_{12}V_{13}$; these turn out to be in one-to-one correspondence with \emph{Hopf $*$-algebra morphisms} in $\Mor\bigl(\C_0^\uu(\QH),\C_0^\uu(\QG)\bigr)$, i.e.~elements of $\Mor\bigl(\C_0^\uu(\QH),\C_0^\uu(\QG)\bigr)$ intertwining the respective coproducts, and should be thought of as morphisms from $\QG$ to $\QH$. Bicharacters also lift to the universal level; in particular the multiplicative unitary of $\QG$, representing the identity map on $\QG$, has its universal version $\WW^{\QG}\in\M\bigl(\C_0^\uu(\hh{\QG})\tens\C_0^\uu(\QG)\bigr)$ and obvious semi-universal versions $\wW^\QG$ and $\Ww^\QG$ (cf.~\cite{DKSS}). Note finally that if a bicharacter $V$ describes a morphism from $\QG$ to $\QH$, then $\flip(V)^*$ represents its \emph{dual morphism}, from $\hh{\QH}$ to $\hh{\QG}$. All these facts can be found in \cite{MRW} (see also \cite{DKSS}). We say that $\QH$ \emph{is a closed quantum subgroup of $\QG$ in the sense of Vaes} if there exists an injective von Neumann algebra morphism $\gamma\colon\Linf(\hh{\QH})\to\Linf(\hh{\QG})$ intertwining the respective coproducts (one can then show the existence of a morphism from $\QG$ to $\QH$, whose dual is naturally related to the map $\gamma$; moreover the corresponding Hopf $*$-algebra morphism is a surjection from $\C_0^\uu(\QG)$ onto $\C_0^\uu(\QH)$). As in this article we will not need any other notion of a closed quantum subgroup (see \cite{DKSS} for an extended discussion), instead of the longer phrase `a closed quantum subgroup of $\QG$ in the sense of Vaes' we will simply say `a closed quantum subgroup' or even simply `a closed subgroup'.

If $\sN$ is a von Neumann subalgebra of $\Linf(\QG)$ (by which we always understand in particular that $\I_N = \I_{\Linf(\QG)}$), then we say that $\sN$ is
\begin{enumerate}
\item \emph{central} if $\sN$ is contained in the center of $\Linf(\QG)$,
\item a \emph{left coideal} if $\Delta_\QG(\sN)\subset\Linf(\QG)\vtens\sN$,
\item a \emph{right coideal} if $\Delta_\QG(\sN)\subset\sN\vtens\Linf(\QG)$,
\item an \emph{invariant subalgebra} if $\Delta_\QG(\sN)\subset\sN\vtens\sN$,
\item a \emph{Baaj-Vaes} subalgebra if $\sN$ is an invariant subalgebra, $R(\sN)=\sN$ and for any $t\in\RR$ we have $\tau_t(\sN)=\sN$.
\end{enumerate}

It follows from \cite[Proposition 10.5]{BV}, which we will refer to as the \emph{Baaj-Vaes theorem}, that there is a bijective correspondence between subgroups of $\QG$ closed in the sense of Vaes and Baaj-Vaes subalgebras of $\Linf(\hh{\QG})$.

Let us also recall the notion of a \emph{normal coideal} from \cite{ext}. A left coideal $\sL\subset\Linf(\QG)$ is \emph{normal} if
\[
{\ww^{\hh{\QG}}}^*(\sL\tens\I)\ww^{\hh{\QG}}\subset\sL\vtens\Linf(\hh{\QG})
\]
or, equivalently,
\[
\ww^{\QG}(\I\tens\sL){\ww^{\QG}}^*\subset\Linf(\hh{\QG})\vtens\sL.
\]
Note that in the case when $\QG$ is a compact quantum group the notion of normality of a coideal in $\Linf(\QG)$ had been introduced earlier under the names of \emph{coaction symmetry} \cite{tomatsu} or simply \emph{symmetry} \cite{Salmi}.

Finally we will be using the notion of a \emph{co-dual coideal} studied in \cite{KaspSol} (with the idea dating back to \cite{Enock}). Given a left coideal $\sN\subset\Linf(\QG)$ its co-dual object $\dd{\sN}$ is defined as the relative commutant $\sN'\cap\Linf(\hh{\QG})$. It can be shown that $\dd{\sN}$ is a left coideal in $\Linf(\hh{\QG})$ and that $\dd{\dd{\sN}}=\sN$.

\section{The center of a locally compact quantum group}\label{centre}

In this section we introduce a notion of the center $\cZ(\QG)$ of a locally compact quantum group $\QG$ (in particular, $\cZ(\QG)$ will be a closed quantum subgroup of $\QG$) and discuss its first properties. Given a locally compact quantum group $\QG$ we will use the symbol $\hh{\sZ}$ to denote the center of $\Linf(\hh{\QG})$.

\begin{proposition}
Let $\QG$ be a locally compact quantum group. Then the set of central Baaj-Vaes subalgebras of $\Linf(\hh{\QG})$ has a largest element.
\end{proposition}

\begin{proof}
This is quite obvious: the set of central Baaj-Vaes subalgebras of $\hh{\sZ}$ is non-empty, as it contains $\CC\I$. Its largest element is the von Neumann subalgebra generated by the union of all such subalgebras. Clearly it is contained in $\hh{\sZ}$ and is itself a Baaj-Vaes subalgebra.
\end{proof}

\begin{definition}\label{DefZG}
The \emph{center} of the locally compact quantum group $\QG$ is the closed subgroup $\cZ(\QG)$ of $\QG$ such that $\Linf\bigl(\hh{\cZ(\QG)}\bigr)$ is the largest central Baaj-Vaes subalgebra of $\Linf(\hh{\QG})$.
\end{definition}

The center of a locally compact quantum group $\QG$ is the largest subgroup whose ``inclusion morphism'' is \emph{central}, as defined by \cite{wangsimple} (see also \cite{Patri} and in particular \cite{Chirva} where, more generally, central actions are discussed).

\begin{definition}
Let $\QG$ and $\QH$ be locally compact quantum groups and let $\pi\in\Mor\bigl(\C_0^\uu(\QG),\C_0^\uu(\QH)\bigr)$ be a Hopf $*$-homomorphism (\cite{MRW}). We say that $\pi$ is \emph{central} if $(\id\tens\pi)\comp\Delta_\QG^\uu=(\id\tens\pi)\comp\flip\comp\Delta_\QG^\uu$. A closed subgroup $\QH$ of $\QG$ is \emph{central} if the corresponding surjection $\pi_\QH\colon\C_0^\uu(\QG)\to\C_0^\uu(\QH)$ is central.
\end{definition}

The next theorem makes precise the statement that $\cZ(\QG)$ is the largest central subgroup of $\QG$.

\begin{theorem}\label{thmZG}
\noindent
\begin{enumerate}
\item\label{centr1}\sloppy
Let $\QH$ be a central subgroup of $\QG$. Then the range of the corresponding inclusion $\gamma\colon\Linf(\hh{\QH})\hookrightarrow\Linf(\hh{\QG})$ is contained in $\hh{\sZ}$.
\item\label{centr2} The morphism $\pi_\QH$ factors through $\pi_{\cZ(\QG)}\colon\C_0^\uu(\QG)\to\C_0^\uu\bigl(\cZ(\QG)\bigr)$.
\end{enumerate}
\end{theorem}

\begin{proof}
Since $\QH$ is a closed subgroup of $\QG$, we have the normal inclusion $\gamma\colon\Linf(\hh{\QH})\hookrightarrow\Linf(\hh{\QG})$. Let $V\in\M\bigl(\C_0(\hh{\QG})\tens\C_0(\QH)\bigr)$ be the corresponding bicharacter. The results of \cite{DKSS} show that we have
\begin{equation}\label{Vpi}
V=(\Lambda_{\hh{\QG}}\tens\Lambda_{\QH})(\id\tens\pi_\QH)(\WW^\QG),
\end{equation}
and
\begin{equation}\label{Vim}
\gamma\bigl(\Linf(\hh{\QH})\bigr)=\linW\bigl\{(\id\tens\omega)(V)\st\omega\in\B(\Ltwo(\QH))_*\bigr\}.
\end{equation}
Now $\QH$ is central, i.e.~$(\id\tens\pi_\QH)\comp\Delta_\QG^\uu=(\id\tens\pi_\QH)\comp\flip\comp\Delta_\QG^\uu$, so applying both these maps to the second leg of $\wW^\QG$ yields
\begin{equation}\label{VW}
\wW_{12}^\QG[(\id\tens\pi_\QH)\wW^\QG]_{13}=[(\id\tens\pi_\QH)\wW^\QG]_{13}\wW_{12}^\QG.
\end{equation}
Now we apply $(\id\tens\Lambda_{\QG}\tens\Lambda_{\QH})$ to both sides of \eqref{VW} and by \eqref{Vpi} we find that
\[
{\ww_{12}^\QG}V_{13}=V_{13}\ww_{12}^\QG.
\]
In view of \eqref{Vim}, this proves that the image of $\gamma$ is a central subalgebra of $\Linf(\hh{\QG})$ which establishes \eqref{centr1}.

Statement \eqref{centr2} follows from Lemma \ref{HHG} below.
\end{proof}

\begin{lemma}\label{HHG}
Let $\QH_1$ and $\QH_2$ be closed quantum subgroups of a locally compact quantum group $\QG$ with corresponding inclusions and surjections
\[
\arraycolsep=0.1em
\left.\begin{array}{rl}
\gamma_i&\colon\Linf(\hh{\QH_i})\hookrightarrow\Linf(\hh{\QG}),\\
\pi_{\QH_i}&\in\Mor\bigl(\C_0^\uu(\QG),\C_0^\uu(\QH_i)\bigr),\rule{0pt}{3Ex}
\end{array}\right\}
\qquad\qquad{i}=1,2.
\]
Assume that the range of $\gamma_2$ is contained in the range $\gamma_1$. Then $\QH_2$ is a closed subgroup of $\QH_1$ and the corresponding morphism
$\theta\in\Mor\bigl(\C_0^\uu(\QH_1),\C_0^\uu(\QH_2)\bigr)$ satisfies
\begin{equation}\label{zgod}
\pi_{\QH_2}=\theta\comp\pi_{\QH_1}.
\end{equation}
\end{lemma}

\begin{proof}
The fact that $\QH_2$ is a closed subgroup of $\QH_1$ follows immediately from \cite[Theorem 3.3]{DKSS}. Indeed, defining $\lambda=\gamma_1^{-1}\comp\gamma_2$ we obtain a normal inclusion $\Linf(\hh{\QH_2})\hookrightarrow\Linf(\hh{\QH_1})$ commuting with comultiplications. The bicharacters corresponding to the inclusion homomorphisms for $\QH_1\subset\QG$, $\QH_2\subset\QG$ and $\QH_2\subset\QH_1$ are respectively
\[
\begin{split}
V=(\gamma_1\tens\id)(\ww^{\QH_1})\in\Linf(\hh{\QG})\vtens\Linf(\QH_1),\\
U=(\gamma_2\tens\id)(\ww^{\QH_2})\in\Linf(\hh{\QG})\vtens\Linf(\QH_2),\\
T=(\lambda\tens\id)(\ww^{\QH_2})\in\Linf(\hh{\QH_1})\vtens\Linf(\QH_2).
\end{split}
\]
Note that $(\gamma_1\tens\id)T=U$.

\sloppy
The unitary $T$ is a bicharacter, so $(\Delta_{\hh{\QH_1}}\tens\id)T=T_{23}T_{13}$ which can be rewritten as $(\ww^{\QH_1}_{12})^*T_{23}\ww^{\QH_1}_{12}=T_{13}T_{23}$ or
\begin{equation}\label{TWT}
T_{23}\ww^{\QH_1}_{12}=\ww^{\QH_1}_{12}T_{13}T_{23}.
\end{equation}
Applying $(\gamma_1\tens\id\tens\id)$ to both sides of \eqref{TWT} and using multiplicativity of $\gamma_1$ we obtain
\[
T_{23}V_{12}=V_{12}U_{13}T_{23}
\]
which means precisely that the bicharacter $U$ is the composition of $V$ and $T$ according to \cite[Definition 3.5]{MRW}. By \cite[Theorem 4.8]{MRW} the corresponding Hopf $*$-homomorphisms $\pi_{\QH_1},\pi_{\QH_2}$ and $\theta$ satisfy \eqref{zgod}.
\end{proof}

\begin{remark}\label{remZG}
Let us note here some straightforward consequences of the definition and Theorem \ref{thmZG}:
\begin{itemize}
\item the quantum group $\cZ(\QG)$ is abelian (this follows from the fact that $\Linf\bigl(\hh{\cZ(\QG)}\bigr)$ is commutative),
\item if $\QG$ is a classical group then $\cZ(\QG)$ coincides with the classical center of $\QG$,
\item if $\QG$ is abelian then $\cZ(\QG)=\QG$;
\item if $\QG$ is compact, then the notion of $\cZ(\QG)$ introduced here coincides with those considered in Section 6 of \cite{Patri} and Section 2 of \cite{Chirva} (this follows from Theorem \ref{thmZG}),
\item if $\QG$ is such that $\Linf(\hh{\QG})$ is a factor then clearly $\cZ(\QG)$ is trivial ($\Linf\bigl(\hh{\cZ(\QG)}\bigr)$ is equal to $\CC\I$). This is the case e.g.~for $\QG$ one of the quantum ``$az+b$'' groups or the quantum ``$ax+b$'' groups (\cite{azb,nazb,axb}) as well as in many other examples (see e.g.~\cite{FimaF}).
\end{itemize}
\end{remark}

\begin{remark}\label{remZG2}
If $\QG$ is a compact quantum group then it is easily seen from Theorem \ref{thmZG} and results of Bichon, Neshveyev and Yamashita \cite[Lemmas 1.1 \& 1.3]{bny} that $\cZ(\QG)$ is the dual of the (discrete) \emph{chain group} $\mathrm{Ch}\bigl(\mathrm{Rep}(\QG)\bigr)$ of the representation category $\mathrm{Rep}(\QG)$ of $\QG$, i.e.~the universal group with a map from the set of equivalence classes of representations of $\QG$ with the property that $[U]$ is mapped to $[V][W]$ if $U$ is a subrepresentation of the tensor product $V\tp{W}$. This generalizes the results of \cite{Mug} and, more importantly, shows that the center of a compact quantum group is uniquely determined by its fusion semiring. In particular for any Drinfeld-Jimbo deformation $\QG_q$ of a compact semisimple Lie group $G$ (cf.~Section \ref{DJ}) we have $\cZ(G_q)\cong\cZ(G)$. This theme is developed more thoroughly in \cite{CHK}.
\end{remark}

More examples of determining the center of a locally compact quantum group will be given in Section \ref{examples}.

\begin{remark}\label{qbycentral}
Let $\QH$ be a closed quantum subgroup of a locally compact quantum group $\QG$ with corresponding inclusion $\gamma\colon\Linf(\hh{\QH})\hookrightarrow\Linf(\hh{\QG})$. It is easy to see that $\QH$ is central if and only if $\gamma\bigl(\Linf(\hh{\QH})\bigr)$ is a central subalgebra in $\Linf(\hh{\QG})$ (cf.~Theorem \ref{thmZG}). It is also clear that if $\QH$ is central then $\gamma\bigl(\Linf(\hh{\QH})\bigr)\subset\Linf(\hh{\QG})$ is a \emph{normal coideal} in the sense of \cite[Definition 4.1]{ext} (cf.~Definition \ref{DefNormSubgrp}). It follows that it is \emph{strongly normal} (\cite[Definition 4.5 \& Theorem 4.6]{ext}). Thus the co-dual of $\gamma\bigl(\Linf(\hh{\QH})\bigr)$ is a Baaj-Vaes subalgebra of $\Linf(\QG)$ and writing $\QG/\QH$ for the corresponding locally compact quantum group we get the short exact sequence $\xymatrix@1@C-2ex{\{e\}\ar[r]&\QH\ar[r]&\QG\ar[r]&\QG/\QH\ar[r]&\{e\}}$.
\end{remark}

Applying the conclusion of Remark \ref{qbycentral} to the subgroup $\cZ(\QG)$ of a locally compact quantum group $\QG$ we obtain the locally compact quantum group $\QG/\!\cZ(\QG)$ and short exact sequence
\begin{equation}\label{exactSequence}
\xymatrix{\{e\}\ar[r]&\cZ(\QG)\ar[r]&\QG\ar[r]&\QG/\!\cZ(\QG)\ar[r]&\{e\}.}
\end{equation}
In what follows we will analyze this short exact sequence. Note that, similarly to the compact quantum group case studied earlier by Patri and Chirvasitu, also here center is defined as a certain universal object and $\Linf\bigl(\cZ(\QG)\bigr)$ does not in general admit a concrete description. In the compact case the quantum group $\QG/\!\cZ(\QG)$ corresponds to what is called in \cite{Chirva} a \emph{cocenter} of $\QG$.

\begin{lemma}\label{altYZlemma}
\newcounter{c}
Let $\QG$ be a locally compact quantum group and let $\sX$ and $\sY$ be von Neumann subalgebras of $\Linf(\QG)$. Then if either
\begin{enumerate}
\item\label{aYZ1} $\Delta_\QG(\sX)\subset\sX\vtens\Linf(\QG)$ and $\Delta_\QG(\sX)\subset\sY\vtens\Linf(\QG)$
\setcounter{c}{\value{enumi}}
\end{enumerate}
or
\begin{enumerate}
\setcounter{enumi}{\value{c}}
\item\label{aYZ2} $\Delta_\QG(\sX)\subset\Linf(\QG)\vtens\sX$ and $\Delta_\QG(\sX)\subset\Linf(\QG)\vtens\sY$
\end{enumerate}
then $\sX\subset\sY$.
\end{lemma}

\begin{proof}
Under assumption \eqref{aYZ1} the restriction $\bigl.\Delta\bigr|_{\sX}$ becomes a (right) action of $\QG$ on $\sX$. Therefore, by \cite[Corollary 2.7]{proj} we have
\[
\sX=\linW\bigl\{(\id\tens\omega)\Delta(x)\st{x}\in\sX,\:\omega\in\B(\Ltwo(\QG))_*\bigr\}.
\]
Since this set is contained in $\sY$ (by assumption) we have $\sX\subset\sY$.

In case of assumption \eqref{aYZ2}, the subalgebra $R(\sX)$ is a right coideal, so again
\[
R(\sX)=\linW\bigl\{(\id\tens\omega)\Delta(R(x))\st{x}\in\sX,\:\omega\in\B(\Ltwo(\QG))_*\bigr\}.
\]
It follows that $R(\sX)\subset{R(\sY)}$, so $\sX\subset\sY$.
\end{proof}

\begin{theorem}\label{LhatGthm}
The largest central Baaj-Vaes subalgebra of $\Linf(\hh{\QG})$ coincides with
\begin{equation}\label{defLhatG}
\hh{\sL}(\QG)=\bigl\{y\in\Linf(\hh{\QG})\st\Delta_{\hh{\QG}}(y)\in\Linf(\hh{\QG})\vtens\hh{\sZ}\bigr\}.
\end{equation}
\end{theorem}

\begin{proof}
It is clear that for $x\in\Linf\bigl(\hh{\cZ(\QG)}\bigr)$ (i.e.~in the largest central Baaj-Vaes subalgebra of $\Linf(\hh{\QG})$) we have
\[
\Delta_{\hh{\QG}}(x)\in\Linf\bigl(\hh{\cZ(\QG)}\bigr)\vtens\Linf\bigl(\hh{\cZ(\QG)}\bigr)\subset\hh{\sZ}\vtens\hh{\sZ}\subset\Linf(\hh{\QG})\vtens\hh{\sZ},
\]
so $\Linf\bigl(\hh{\cZ(\QG)}\bigr)\subset\hh{\sL}(\QG)$.

The set $\hh{\sL}(\QG)$ is a left coideal in $\Linf(\hh{\QG})$ because if $y\in\hh{\sL}(\QG)$ then for any $\omega\in\B\bigl(\Ltwo(\QG)\bigr)_*$
\[
\begin{split}
\Delta_{\hh{\QG}}\bigl((\omega\tens\id)\Delta_{\hh{\QG}}(y)\bigr)&=(\omega\tens\id\tens\id)\bigl((\id\tens\Delta_{\hh{\QG}})\Delta_{\hh{\QG}}(y)\bigr)\\
&=(\omega\tens\id\tens\id)\bigl((\Delta_{\hh{\QG}}\tens\id)\Delta_{\hh{\QG}}(y)\bigr)\\
&\in(\omega\tens\id\tens\id)\bigl((\Delta_{\hh{\QG}}\tens\id)\Delta_{\hh{\QG}}\bigl(\hh{\sL}(\QG)\bigr)\bigr)\\
&\subset(\omega\tens\id\tens\id)\bigl((\Delta_{\hh{\QG}}\tens\id)\bigl(\Linf(\hh{\QG})\vtens\hh{\sZ}\bigr)\bigr)\\
&\subset(\omega\tens\id\tens\id)\bigl(\Linf(\hh{\QG})\vtens\Linf(\hh{\QG})\vtens\hh{\sZ}\bigr)\\
&\subset\Linf(\hh{\QG})\vtens\hh{\sZ},
\end{split}
\]
so $(\omega\tens\id)\Delta_{\hh{\QG}}(y)\in\hh{\sL}(\QG)$. It is also immediately clear from \eqref{defLhatG} that $\hh{\sL}(\QG)$ is $\tau$-invariant.

Now Lemma \ref{altYZlemma}(2), with $\QG$ replaced by $\hh{\QG}$, $\sX=\hh{\sL}(\QG)$ and $\sY=\hh{\sZ}$ shows that $\hh{\sL}(\QG)\subset\hh{\sZ}$.

Let us show that for $y\in\hh{\sL}(\QG)$ we actually have
\[
\Delta_{\hh{\QG}}(y)\in\hh{\sZ}\vtens\hh{\sZ}.
\]
Indeed, take $y\in\hh{\sL}(\QG)$ and let us write $\hh{\ww}$ for $\ww^{\hh{\QG}}$. We have
\[
\begin{split}
\hh{\ww}_{12}^*\Delta_{\hh{\QG}}(y)_{23}\hh{\ww}_{12}&=\hh{\ww}_{12}^*\hh{\ww}_{23}(\I\tens{y}\tens\I)\hh{\ww}_{23}^*\hh{\ww}_{12}\\
&=\hh{\ww}_{12}^*\hh{\ww}_{23}\hh{\ww}_{12}(\I\tens{y}\tens\I)\hh{\ww}_{12}^*\hh{\ww}_{23}^*\hh{\ww}_{12}\\
&=\hh{\ww}_{13}\hh{\ww}_{23}(\I\tens{y}\tens\I)\hh{\ww}_{23}^*\hh{\ww}_{13}^*\\
&=\hh{\ww}_{13}\bigl(\I\tens\Delta_{\hh{\QG}}(y)\bigr)\hh{\ww}_{13}^*\\
&=\I\tens\Delta_{\hh{\QG}}(y)=\Delta_{\hh{\QG}}(y)_{23},
\end{split}
\]
where in the second step we used the fact that $y\in\hh{\sZ}$ and in the fifth that the right leg of $\Delta_{\hh{\QG}}(y)$ is in $\hh{\sZ}$. Slicing the equality
\[
\Delta_{\hh{\QG}}(y)_{23}\hh{\ww}_{12}= \hh{\ww}_{12}\Delta_{\hh{\QG}}(y)_{23}
\]
over the first and third leg we find that the left leg of $\Delta_{\hh{\QG}}(y)$ is in the center of $\Linf(\hh{\QG})$.

It follows that we can rewrite \eqref{defLhatG} in the following way
\begin{equation}\label{defLhatG2}
\hh{\sL}(\QG)=\bigl\{y\in\Linf(\hh{\QG})\st\Delta_{\hh{\QG}}(y)\in\hh{\sZ}\vtens\hh{\sZ}\bigr\}.
\end{equation}
It is now easy to see that $\hh{\sL}(\QG)$ is a Baaj-Vaes subalgebra and we already know that it is contained in $\hh{\sZ}$. Clearly it is the largest central Baaj-Vaes subalgebra because any element $z$ of a central Baaj-Vaes subalgebra must necessarily satisfy
\[
\Delta_{\hh{\QG}}(z)\in\hh{\sZ}\vtens\hh{\sZ}
\]
so $z$ belongs to $\hh{\sL}(\QG)$.
\end{proof}

\begin{remark}
The object $\hh{\sL}(\QG)$ can be equivalently defined as the largest left coideal in $\Linf(\hh{\QG})$ contained in $\hh{\sZ}$. Indeed, we showed in the proof of Theorem \ref{LhatGthm} that $\hh{\sL}(\QG)$ is a left coideal contained in $\hh{\sZ}$. Let $\sL$ be a left coideal of $\Linf(\hh{\QG})$ contained in $\hh{\sZ}$. Then for $z\in\sL$ we have
\[
\Delta_{\hh{\QG}}(z)\in\Linf(\hh{\QG})\vtens\sL\subset\Linf(\hh{\QG})\vtens\hh{\sZ},
\]
so $z\in\hh{\sL}(\QG)$.
\end{remark}

\section{The quantum group $\Inn(\QG)$}\label{Inner}

In the previous section we introduced the quotient quantum group $\QG/\!\cZ(\QG)$ completing the exact sequence \eqref{exactSequence}. By definition $\Linf\bigl(\QG/\!\cZ(\QG)\bigr)$ is the co-dual of the coideal $\Linf\bigl(\hh{\cZ(\QG)}\bigr)\subset\Linf(\hh{\QG})$. In this section we provide another, explicit description of $\Linf\bigl(\QG/\!\cZ(\QG)\bigr)$, to be identified with $\Linf\bigl(\Inn(\QG)\bigr)$.

\begin{theorem}\label{MGZG}
Let $\sM$ be the von Neumann subalgebra of $\Linf(\QG)$ generated by
\begin{equation}\label{genBy}
\bigl\{(\omega\tens\id)(\ww(x\tens\I)\ww^*)\st{x}\in\Linf(\hh{\QG}),\:\omega\in\B(\Ltwo(\QG))_*\bigr\}.
\end{equation}
Then $\sM=\Linf\bigl(\QG/\!\cZ(\QG)\bigr)$.
\end{theorem}

\begin{proof}
Let us first show that $\sM$ is a left coideal in $\Linf(\QG)$. For any $x\in\Linf(\hh{\QG})$ and $\omega\in\B\bigl(\Ltwo(\QG)\bigr)_*$ we have
\begin{align*}
\Delta_\QG\bigl((\omega\tens\id)(\ww(x\tens\I)\ww^*)\bigr)&= (\omega\tens\id\tens\id)\bigl( \id \tens\Delta_{\QG} (\ww(x\tens\I)\ww^*)\bigr)
\\&=(\omega\tens\id\tens\id)\bigl(\ww_{12}\ww_{13}(x\tens\I\tens\I)\ww^*_{13}\ww_{12}^*\bigr).
\end{align*}
Now, to see that the second leg of $\Delta_\QG\bigl((\omega\tens\id)(\ww(x\tens\I)\ww^*)\bigr)$ belongs to $\sM$, slice the left leg of the above equality with $\eta\in\B\bigl(\Ltwo(\QG)\bigr)_*$:
\[
(\eta\tens\id)\Delta_\QG\bigl((\omega\tens\id)(\ww(x\tens\I)\ww^*)\bigr)
=(\omega\tens\eta\tens\id)\bigl(\ww_{12}\ww_{13}(x\tens\I\tens\I)\ww^*_{13}\ww_{12}^*\bigr).
\]
It follows that the set
\[
\bigl\{(\eta\tens\id)\Delta_\QG\bigl((\omega\tens\id)(\ww(x\tens\I)\ww^*)\bigr)\st{x}\in\Linf(\hh{\QG}),\:\omega,\eta\in\B\bigl(\Ltwo(\QG)\bigr)_*\bigr\}
\]
coincides with
\[
\begin{split}
\bigl\{(\omega\tens\eta\tens\id)&\bigl(\ww_{12}\ww_{13}(x\tens\I\tens\I)\ww^*_{13}\ww_{12}^*\bigr)
\st{x}\in\Linf(\hh{\QG}),\:\omega,\eta\in\B\bigl(\Ltwo(\QG)\bigr)_*\bigr\}\\
&=\bigl\{(\omega'\tens\eta'\tens\id)\bigl(\ww_{13}(x\tens\I\tens\I)\ww^*_{13}\st{x}\in\Linf(\hh{\QG}),\:\omega',\eta'\in\B\bigl(\Ltwo(\QG)\bigr)_*\bigr\}\subset\sM.
\end{split}
\]

Let us now determine the co-dual of $\sM$. An element $y\in\Linf(\hh{\QG})$ belongs to $\dd{\sM}$ if and only if for any $x\in\Linf(\hh{\QG})$ we have
\[
\ww(x\tens\I)\ww^*(\I\tens{y})=(\I\tens{y})\ww(x\tens\I)\ww^*
\]
which is the same as saying that $\ww^*(\I\tens{y})\ww$ commutes with all elements of the form $x\tens\I$ with $x\in\Linf(\hh{\QG})$.

Note that $\ww^*(\I\tens{y})\ww=\flip\bigl(\Delta_{\hh{\QG}}(y)\bigr)$, and so it belongs to $\Linf(\hh{\QG})\vtens\Linf(\hh{\QG})$, and therefore we conclude that $y\in\Linf(\hh{\QG})$ is in $\dd{\sM}$ if and only if $\ww^*(\I\tens{y})\ww\in\hh{\sZ}\vtens\Linf(\hh{\QG})$ or, in other words,
\[
\dd{\sM}=\bigl\{y\in\Linf(\hh{\QG})\st\Delta_{\hh{\QG}}(y)\in\Linf(\hh{\QG})\vtens\hh{\sZ}\bigr\}.
\]
We have shown in Theorem \ref{defLhatG} that this is precisely the subalgebra $\Linf\bigl(\hh{\cZ(\QG)}\bigr)$ of $\Linf(\hh{\QG})$. It follows that
\[
\sM=\dd{\dd{\sM}}=\Linf\bigl(\QG/\!\cZ(\QG)\bigr).
\]
\end{proof}

\begin{corollary}\label{NCor}
The von Neumann algebra $\Linf\bigl(\QG/\!\cZ(\QG)\bigr)$ can also be described as the von Neumann subalgebra of $\Linf(\QG)$ generated by
\begin{equation}\label{eN}
\bigl\{(\omega\tens\id)(\ww^*(x\tens\I)\ww)\st{x}\in\Linf(\hh{\QG}),\:\omega\in\B(\Ltwo(\QG))_*\bigr\}.
\end{equation}
\end{corollary}

\begin{proof}
Denote by $\sN$ the the von Neumann subalgebra of $\Linf(\QG)$ generated by \eqref{eN}. For any $x\in\Linf(\hh{\QG})$ and $\omega\in\B\bigl(\Ltwo(\QG)\bigr)_*$ we have
\[
\begin{split}
\hh{J}\bigl((\omega\tens\id)\bigl(\ww(x\tens\I)\ww^*\bigr)\bigr) \hh{J}
&=(J\omega{J}\tens\id)\bigl((J\tens\hh{J})\ww(x\tens\I)\ww^*(J\tens\hh{J})\bigr)\\
&=(J\omega{J}\tens\id)\bigl(\ww^*(JxJ\tens\I)\ww\bigr)
\end{split}
\]
(cf.~\cite[Equation(5.22)]{mnw}). It follows that $\sN=R(\sM)=R\bigl(\Linf\bigl(\QG/\!\cZ(\QG)\bigr)\bigr)=\Linf\bigl(\QG/\!\cZ(\QG)\bigr)$.
\end{proof}

It follows immediately from Theorem \ref{MGZG} that the von Neumann subalgebra of $\Linf(\QG)$ generated by \eqref{genBy} is a Baaj-Vaes subalgebra. The proof of this fact given above is short and simple. We would like to stress, however, that one can prove this without referring to $\cZ(\QG)$ at all. More precisely, one can easily establish that, using the notation of the proof of Corollary \ref{NCor}, we have $\sN=R(\sM)$. But we can also show that actually $\sN=\sM$: take a selfadjoint $y\in\dd{\sM}$. Using the fact that $\dd{\sM}\subset\hh{\sZ}$ we get $\hh{J}y\hh{J}=y$ and compute
\[
\begin{split}
\ww(\I\tens{y})\ww^*&=\ww(J\tens\hh{J})(\I\tens\hh{J}y\hh{J})(J\tens\hh{J})\ww^*\\
&=(J\tens\hh{J})\ww^*(\I\tens{y})\ww(J\tens\hh{J})\\
&\quad\in(J\tens\hh{J})\bigl(\hh{\sZ}\vtens\Linf(\hh{\QG})\bigr)(J\tens\hh{J})\\
&\quad\subset\hh{\sZ}\vtens\Linf(\hh{\QG})'
\subset\hh{\sZ}\vtens\B\bigl(\Ltwo(\QG)\bigr)
\end{split}
\]
as the map $\Linf(\hh{\QG})\ni{u}\mapsto{J}u{J}\in\Linf(\hh{\QG})$ preserves $\hh{\sZ}$ (this map is the composition $\hh{R}\comp*$). It follows that each selfadjoint element of $\dd{\sM}$ belongs to $\dd{\sN}$, and hence
\begin{equation}\label{zaw}
\dd{\sM}\subset\dd{\sN}.
\end{equation}

Applying co-duality to both sides of \eqref{zaw} we obtain $\sM\supset\sN$ and, as $\sN=R(\sM)$, we have
\[
\sM\supset{R(\sM)}.
\]
Since $R^2=\id$, this shows that $\sM=R(\sM)=\sN$.

Thus, given a locally compact quantum group $\QG$, the von Neumann subalgebra $\sM$ of $\Linf(\QG)$ generated by
\[
\bigl\{(\omega\tens\id)(\ww(x\tens\I)\ww^*)\st{x}\in\Linf(\hh{\QG}),\:\omega\in\B(\Ltwo(\QG))_*\bigr\}
\]
is a Baaj-Vaes subalgebra. By the Baaj-Vaes theorem (cf.~Section \ref{intro}) the algebra $\sM$ with comultiplication $\bigl.\Delta_\QG\bigr|_{\sM}$ becomes an algebra of functions on a locally compact quantum group.

\begin{definition}\label{DefInnG}
Let $\QG$ be a locally compact quantum group and let $\sM\subset\Linf(\QG)$ be the subalgebra generated by
\[
\bigl\{(\omega\tens\id)(\ww(x\tens\I)\ww^*)\st{x}\in\Linf(\hh{\QG}),\:\omega\in\B(\Ltwo(\QG))_*\bigr\}.
\]
We define the locally compact quantum group $\Inn(\QG)$ by setting
\[
\Linf\bigl(\Inn(\QG)\bigr)=\sM\quad\text{and}\quad\Delta_{\Inn(\QG)}=\bigl.\Delta_{\QG}\bigr|_{\sM}.
\]
\end{definition}

In view of Theorem \ref{MGZG} we immediately conclude that $\Inn(\QG)=\QG/\!\cZ(\QG)$. Nevertheless we would like to keep the double terminology because, as we will see later on, the natural interpretation of $\Inn(\QG)$ is that of the group of inner automorphisms of $\QG$ (see Section \ref{InnAct}).

Thus, from \eqref{exactSequence} we immediately obtain existence of the following short exact sequence:
\[
\xymatrix{\{e\}\ar[r]&\cZ(\QG)\ar[r]&\QG\ar[r]&\Inn(\QG)\ar[r]&\{e\}.}
\]

\begin{remark}
An important point is to note that equality of $\Inn(\QG)$ and $\QG/\!\cZ(\QG)$ shows that the object defined by a universal property (namely $\QG/\!\cZ(\QG)$ defined via the universal property of $\cZ(\QG)$) has a much more concrete description given by our definition of $\Inn(\QG)$. This can be very helpful in concrete cases (see e.g.~Remark \ref{projVer}). We are grateful to Alexandru Chirvasitu for pointing out to us that for compact quantum groups this result can be also deduced from \cite[Proposition 2.13 (b)]{Chirva}.
\end{remark}

\subsection{$\Inn(\QG)$ for $\QG$ a compact quantum group}

In this subsection we let $\QG$ be a compact quantum group. We will denote by $\Irr(\QG)$ the set of equivalence classes of irreducible unitary representations of $\QG$ and for each $\iota\in\Irr(\QG)$ we will choose a representative $u^\iota$ of $\iota$. Then $u^\iota\in{M_{n_\iota}(\CC)}\tens\Pol(\QG)$, where $\Pol(\QG)$ is the Hopf $*$-algebra associated to $\QG$. The direct sum
\[
\bigoplus_{\iota\in\Irr(\QG)}u^\iota
\]
can be interpreted as either $\wW\in\M\bigl(\c0(\hh{\QG})\tens\C^\uu(\QG)\bigr)$ or as $\ww\in\linf(\hh{\QG})\vtens\Linf(\QG)$ depending on whether we view $\Pol(\QG)$ as a subalgebra of $\C^\uu(\QG)$ or $\C(\QG)$ (or $\Linf(\QG)$). Recall that $\hh{\QG}$ is coamenable, so $\wW=\WW$ and $\ww=\Ww$.

\begin{theorem}\label{cptThm}
Let $\QG$ be a compact quantum group and let $\sA$ be the \cst-subalgebra of $\C(\QG)$ generated by the subset
\begin{equation}\label{genA}
\bigl\{(\omega\tens\id)(\ww(x\tens\I)\ww^*)\st{x}\in\c0(\hh{\QG}),\:\omega\in\c0(\hh{\QG})^*\bigr\}.
\end{equation}
Then $\sA$ satisfies $\Delta_\QG(\sA)\subset\sA\tens\sA$ and $\QH$ defined by $\C(\QH)=\sA$ and $\Delta_\QH=\bigl.\Delta_\QG\bigr|_{\C(\QH)}$ is a compact quantum group. Moreover $\QH$ coincides with $\Inn(\QG)$.
\end{theorem}

\begin{proof}
Let us take a matrix unit $e^\iota_{i,j}$ in a direct summand $M_{n_\iota}(\CC)$ of $\c0(\hh{\QG})=\bigoplus\limits_{\mu\in\Irr(\QG)}M_{n_\mu}(\CC)$ and let $\omega^\iota_{k,l}$ be the functional on $\c0(\hh{\QG})$ equal to $1$ on $e^\iota_{k,l}$ and $0$ on all other matrix units (from any block). Then
\[
(\omega^\iota_{k,l}\tens\id)\bigl(\ww(e^\iota_{i,j}\tens\I)\ww^*\bigr)
=(\omega^\iota_{k,l}\tens\id)\sum_{a,b}e_{a,b}^\iota\tens{u_{a,i}^\iota}{u_{b,j}^\iota}^*
={u_{k,i}^\iota}{u_{l,j}^\iota}^*.
\]
It follows that if $\cA$ is the $*$-algebra generated by
\[
\bigl\{u^\iota_{i,j}{u^\iota_{k,l}}^*\st\iota\in\Irr{\QG},1\leq{i},j,k,l\leq{n_\iota}\bigr\}
\]
then $\cA$ is dense in $\sA$. Moreover, $\cA$ is clearly invariant for $\Delta_\QG$ and for the antipode $S_\QG$. It follows that $\Delta_\QG(\sA)(\I\tens\sA)$ and $(\sA\tens\I)\Delta_\QG(\sA)$ are dense in $\sA\tens\sA$ because $\cA\tens\cA$ is invariant for the maps $(a\tens{b})\mapsto\bigl((\id\tens{S})\Delta_\QG(a)\bigr)(\I\tens{b})$ and $(a\tens{b})\mapsto(a\tens\I)\bigl((S\tens\id)\Delta_\QG(b)\bigr)$ which are inverse to $a\tens{b}\mapsto\Delta_\QG(a)(\I\tens{b})$ and $a\tens{b}\mapsto(a\tens\I)\Delta_\QG(b)$ respectively. Therefore $\sA$ has the structure of the algebra of functions on a compact quantum group. Since the Haar measure on $\C(\QG)$ is faithful, its restriction to $\C(\QH)$ is also faithful and so $\Linf(\QH)$ is the weak closure of $\sA$. It follows that $\Linf(\QH)$ is equal to $\Linf\bigl(\Inn(\QG)\bigr)$, and so $\QH=\Inn(\QG)$.
\end{proof}

\begin{remark}\label{projVer}
\noindent\begin{enumerate}
\item Consider $\QG=\mathrm{SU}_q(2)$ (\cite{su2}). Then it is easy to see that $\Inn(\QG)$ is the quantum group $\mathrm{SO}_q(3)$ (see \cite[Section 3]{podles}). In particular $\cZ(\QG)=\mathbb{Z}_2$ (cf.~Remark \ref{remZG2}, this can also be proved using \cite[Proposition 3.4]{podles}).
\item More generally, if $\QG$ is a compact matrix quantum group with irreducible fundamental representation and commutative fusion rules then $\Inn(\QG)$ is the \emph{projective version} of $\QG$ studied in \cite[Section 3]{BanVer}.
\item We thank the anonymous referee for the following observation. In fact when $\QG$ is a compact matrix quantum group with irreducible fundamental representation $u\in\M_n\tens\Pol(\QG)$ a more precise description of $\cZ(\QG)$ and $\Inn(\QG)$ is available: $\Pol\bigl(\cZ(\QG)\bigr)$ is the quotient of $\Pol(\QG)$ by the relations
\[
u_{i,j}=0\text{ if $i\neq{j}$,}\qquad{u_{i,i}}=u_{j,j},\quad{i,j}=1,\dotsc,n
\]
(this follows from irreducibility of $u$ and the universal property of $\cZ(\QG)$). In particular $\cZ(\QG)$ is a subgroup of $\TT$.

Let $g$ be the image of the diagonal matrix elements of $u$ in the epimorphism $\Pol(\QG)\to\Pol\bigl(\cZ(\QG)\bigr)$. Then the algebra $\Pol\bigl(\Inn(\QG)\bigr)$ is the $*$-subalgebra of $\Pol(\QG)$ generated by
\[
\bigl\{u^{e_1}_{i_1,j_1}{\dotsm}u^{e_k}_{{i_k},{j_k}}\st{k\in\NN},\;e_1,\dotsc,e_k\in\{1,*\},\;g^{e_1+\dotsm+e_k}=\I\bigr\}.
\]
Indeed, on one hand elements of the set above are invariant under the action of $\cZ(\QG)$ and on the other hand elements of the form $u^\alpha_{i,j}{u^\alpha_{k,l}}^*$ which generate $\Pol\bigl(\Inn(\QG)\bigr)$ clearly belong to this set.
\end{enumerate}
\end{remark}

\section{Action by inner automorphisms}\label{InnAct}

In this section we discuss natural actions of $\Inn(\QG)$ on the dual quantum group $\hh{\QG}$, which in a sense replace the action of $\Inn(G)$ on $G$ known in the classical setting.

Let $\QG$ be a locally compact quantum group. For any $x\in\Linf(\hh{\QG})$ we define
\begin{equation}\label{alpha}
\alR(x)=\ww(x\tens\I)\ww^*\quad\text{and}\quad\alL(x)=\flip\bigl(\ww^*(x\tens\I)\ww\bigr).
\end{equation}
Note that
\[
\alR(x)\in\Linf(\hh{\QG})\vtens\Linf\bigl(\Inn(\QG)\bigr)\quad\text{and}\quad
\alL(x)\in\Linf\bigl(\Inn(\QG)\bigr)\vtens\Linf(\hh{\QG})
\]
and thus we obtain two mappings
\[
\alR\colon\Linf(\hh{\QG})\to\Linf(\hh{\QG})\vtens\Linf\bigl(\Inn(\QG)\bigr)\quad\text{and}\quad
\alL\colon\Linf(\hh{\QG})\to\Linf\bigl(\Inn(\QG)\bigr)\vtens\Linf(\hh{\QG})
\]
which are easily shown to be actions of $\hh{\QG}$: e.g.~for $x\in\Linf(\hh{\QG})$ we have
\[
\begin{split}
(\id\tens\Delta_{\Inn(\QG)})\bigl(\alR(x)\bigr)&=(\id\tens\Delta_{\QG})\bigl(\alR(x)\bigr)\\
&=\ww_{12}\ww_{13}(x\tens\I\tens\I)\ww_{13}^*\ww_{12}^*\\
&=\ww_{12}\bigl[\ww(x\tens\I)\ww^*\bigr]_{13}\ww_{12}^*\\
&=(\alR\tens\id)\bigl(\ww(x\tens\I)\ww^*\bigr)=(\alR\tens\id)\bigl(\alR(x)\bigr).
\end{split}
\]
yielding $(\id\tens\Delta_{\Inn(\QG)})\comp\alR=(\alR\tens\id)\comp\alR$. Analogous formula $(\Delta_{\Inn(\QG)}\tens\id)\comp\alL=(\id\tens\alL)\comp\alL$ follows from a similar calculation or from the fact that
\[
\alL=\flip\comp(\hh{R}\tens{R})\comp\alR\comp\hh{R}.
\]

Suppose $\QG$ is a classical locally compact group $G$. Then for any $x\in\Linf(\hh{G})=\vN(G)$ (the group von Neumann algebra of $G$) the element $\ww^*(x\tens\I)\ww$ belongs to $\vN(G)\vtens\Linf(G)$, i.e.~it is an essentially bounded (in fact strongly continuous) function $G\to\vN(G)$. It is easy to see that it is the function
\[
G\ni{t}\longmapsto\rho_t^*x\rho_t,
\]
where $t\mapsto\rho_t$ is the right regular representation of $G$. In particular, for $x=\rho_s$ the element $\ww^*(x\tens\I)\ww$ corresponds to the function
\[
G\ni{t}\longmapsto\rho_{t^{-1}st}\in\vN(G).
\]
This justifies the terminology introduced in the following definition.

\begin{definition}
Let $\QG$ be a locally compact quantum group. The actions $\alL$ and $\alR$ of $\Inn(\QG)$ on $\Linf(\hh{\QG})$ introduced by \eqref{alpha} will be called respectively the left and right action by \emph{inner automorphisms}.
\end{definition}

\begin{remark}\label{RemNorm}
Let $\QG$ be a locally compact quantum group and let $\sL$ be a von Neumann subalgebra of $\Linf(\hh{\QG})$. The mapping $\alL\colon\Linf(\hh{\QG})\to\Linf\bigl(\Inn(\QG)\bigr)\vtens\Linf(\hh{\QG})$ can be restricted to $\sL$, but there is no guarantee that its values will lie in $\Linf\bigl(\Inn(\QG)\bigr)\vtens\sL$. In case this holds, we may call the subalgebra $\sL$ \emph{invariant under $\alL$} because $\bigl.\alL\bigr|_{\sL}$ defines a left action of $\Inn(\QG)$ on $\sL$.

In the particular case when $\sL$ is a left coideal the invariance under $\alL$ was studied in \cite[Section 4]{ext} and it was called \emph{normality}. Thus a left coideal of $\Linf(\hh{\QG})$ is normal if and only if it is invariant under the left action by inner automorphisms (cf.~\cite[Definition 4.1]{ext}, see also Section 1).

Note further that invariance under left and right action by inner automorphisms are equivalent conditions for $R$-invariant coideals.
\end{remark}

\begin{definition}[{\cite[Definition 2.10]{VaesVainerman2}}]\label{DefNormSubgrp}
Let $\QH$ be a closed quantum subgroup of a locally compact quantum group $\QG$ (recall this means that $\Linf(\hh{\QH})$ injects in a comultiplication preserving way into $\Linf(\hh{\QG})$). We say that $\QH$ is \emph{normal} in the sense of Vaes if $\Linf(\hh{\QH})\subset\Linf(\hh{\QG})$ is a normal coideal.
\end{definition}

Let us point out that Definition \ref{DefNormSubgrp} is not exactly the same as in \cite{VaesVainerman2} due to certain evolution of the notion of a closed quantum subgroup over time. Still, Remark \ref{RemNorm} clearly shows that a closed quantum subgroup $\QH$ of $\QG$ is normal if and only if $\Linf(\hh{\QH})\subset\Linf(\hh{\QG})$ is invariant under the action by inner automorphisms (left or right).

\subsection{Normal subgroups in the sense of Wang}

Already as early as in 1995 Wang defined the notion of a normal subgroup of a compact quantum group:

\begin{definition}[{\cite[Section 2]{free}}]\label{WangDef}
Let $\KK$ and $\LL$ be a compact quantum groups. We say that $\LL$ is a \emph{normal quantum subgroup} of $\KK$ in the sense of Wang if there is a surjective map $\rho\colon\C(\KK)\to\C(\LL)$ commuting with comultiplications such that for any irreducible representation $u$ of $\KK$ the multiplicity of the trivial representation in $(\id\tens\rho)u$ is either zero or equal to the dimension of $u$.
\end{definition}

\begin{remark}
Subgroups of compact quantum groups have been a topic of investigations for a much longer time than those of locally compact quantum groups (cf.~e.g.~\cite{PodPhd,spheres}). Initially a compact quantum group $\KK$ would most likely be described on the universal level, i.e.~by the \cst-algebra we now call $\C^\uu(\KK)$, or without any restriction on which completion of the associated Hopf $*$-algebra $\Pol(\KK)$ is used (see \cite[Definition 1.1]{pseudogr}). In this context a closed subgroup $\LL$ of $\KK$ was usually taken to be a compact quantum group such that there exists a surjective map $\pi\colon\C(\KK)\to\C(\LL)$ intertwining the respective comultiplications. Here $\C(\KK)$ and $\C(\LL)$ could be \emph{exotic} completions of $\Pol(\KK)$ and $\Pol(\LL)$ in the sense of \cite{dkps}.

However, the map $\pi$ can be lifted to $\pi^\uu\colon\C^\uu(\KK)\to\C^\uu(\LL)$. Indeed, as $\pi$ maps matrix elements of unitary representations to matrix elements of unitary representations, we have $\pi\bigl(\Pol(\KK)\bigr)\subset\Pol(\LL)$, so $\pi$ restricted to $\Pol(\KK)$ can be considered as a map into $\Pol(\LL)\subset\C^\uu(\LL)$. Thus the universal property of $\C^\uu(\KK)$ gives the $*$-homomorphism $\pi^\uu$. Now it is easy to check that the composition of $\pi^\uu$ with the \emph{reducing morphism} $\C^\uu(\LL)\to\C_\textrm{\tiny{r}}(\LL)$ is surjective, so by \cite[Theorem 3.6]{DKSS} $\pi^\uu$ is surjective. It follows that $\LL$ is a closed subgroup of $\KK$ in the sense of Woronowicz and, by \cite[Theorem 6.1]{DKSS}, also in the sense of Vaes. In particular we have a normal unital $*$-homomorphism $\linf(\hh{\LL})\hookrightarrow\linf(\hh{\KK})$.
\end{remark}

\begin{theorem}
Let $\KK$ be a compact quantum group and $\LL$ be its closed subgroup. Then $\KK$ is normal in the sense of Wang if and only if it is normal.
\end{theorem}

\begin{proof}
Let $\pi\colon\C^\uu(\KK)\to\C^\uu(\LL)$ be the epimorphism identifying $\LL$ as a closed subgroup of $\KK$ and define
\[
\begin{split}
\Pol(\KK/\LL)&=\bigl\{a\in\Pol(\KK)\st(\id\tens\pi)\Delta_\KK(a)=a\tens\I\bigr\},\\
\Pol(\LL\backslash\KK)&=\bigl\{a\in\Pol(\KK)\st(\pi\tens\id)\Delta_\KK(a)=\I\tens{a}\bigr\}.
\end{split}
\]
Using \cite[Eq.~(8.6)]{cqg} one can easily show that $R$ maps these subalgebras of $\Pol(\KK)$ bijectively onto one another.
Moreover $\Pol(\KK/\LL)$ is a strongly dense $*$-subalgebra of $\Linf(\KK/\LL)$ which by definition is the co-dual of the coideal $\linf(\hh{\LL})\subset\linf(\hh{\KK})$ and in fact
\begin{equation}\label{cap}
\Pol(\KK/\LL)=\Linf(\KK/\LL)\cap\Pol(\KK)
\end{equation}
(cf.~\cite[Proposition 3.4]{ext}).

By \cite[Proposition 3.2]{wang3} $\LL$ is normal in the sense of Wang if and only if $\Pol(\KK/\LL)=\Pol(\LL\backslash\KK)$ or, in other words, if and only if $\Pol(\KK/\LL)$ is $R$-invariant. This means that $\linf(\hh{\LL})\subset\linf(\hh{\KK})$ is strongly normal (\cite[Definition 4.5]{ext}) which by \cite[Theorem 4.6]{ext} is equivalent to its normality, i.e.~to $\LL$ being normal in the sense of Definition \ref{DefNormSubgrp}.

Thus normality in the sense of Wang implies normality as defined in \cite{ext} (see the introductory section). Conversely, if $\LL$ is normal then $\linf(\hh{\LL})\subset\linf(\hh{\KK})$ is strongly normal, and so $\Linf(\KK/\LL)$ is $R$-invariant. This means that any $a\in\Pol(\KK/\LL)$ satisfies $R(a)\in\Linf(\KK/\LL)$. But the set of all $R(a)$ with $a\in\Pol(\KK/\LL)$ is equal to $\Pol(\LL\backslash\KK)$, so $\Pol(\LL\backslash\KK)\subset\Linf(\KK/\LL)$ which by \eqref{cap} means that $\Pol(\LL\backslash\KK)\subset\Pol(\KK/\LL)$ and since $R^2=\id$ we have $\Pol(\LL\backslash\KK)=\Pol(\KK/\LL)$. The latter is equivalent to Wang-normality of $\LL$ again by \cite[Proposition 3.2]{wang3}.
\end{proof}

Note that in \cite{Patri} it is shown that a certain class of \emph{classical} automorphisms related to inner automorphisms of a compact quantum group preserves every quantum subgroup which is normal in the sense of Wang. In view of the result above and the fact that it is still not clear how to define classical inner automorphisms of a (locally) compact quantum group, so that the definition coincides with the standard one in the classical case, it is natural to ask the following question.

\begin{q}
Every `classical point' of $\Inn(\QG)$ (i.e.~every character of the algebra $\C_0\bigl(\Inn(\QG)\bigr)$) determines an automorphism of a $\cst$-algebra $\C_0(\hh{\QG})$, at least when $\QG$ is compact. How can one characterize the automorphisms of $\C_0(\hh{\QG})$ which arise in this way?
\end{q}

\section{Examples}\label{examples}

In the final short section we discuss centers for cocycle-twists and duals of Drinfeld-Jimbo deformations.

\subsection{Center of a cocycle-twist}

Let $\QG$ be a quantum group and let $\Omega\in\Linf(\hh{\QG})\vtens\Linf(\hh{\QG})$ be a unitary $2$-cocycle, i.e.~a unitary element such that
\[
\Omega_{23}(\id\tens\Delta_{\hh{\QG}})(\Omega)=\Omega_{12}(\Delta_{\hh{\QG}}\tens\id)(\Omega).
\]
Using $\Omega$ we define a map $\Gamma\colon\Linf(\hh{\QG})\to\Linf(\hh{\QG})\vtens\Linf(\hh{\QG})$ via
\[
\Gamma(x)=\Omega\Delta_{\hh{\QG}}(x)\Omega^*,\qquad{x}\in\Linf(\hh{\QG}).
\]
It is easy to see that $\Gamma$ is coassociative and, by results of De Commer (\cite[Section 6]{DeCommer}), $\Linf(\hh{\QG})$ with $\Gamma$ as comultiplication carries the structure of an algebra of functions on a locally compact quantum group which we will denote $\hh{\QG^\Omega}$. Thus in our notation $\Gamma$ becomes $\Delta_{\hh{\QG^\Omega}}$. The dual of $\hh{\QG^\Omega}$ which we will denote by the symbol $\QG^\Omega$ is called the \emph{cocycle twist} of $\QG$ by $\Omega$.

Let us note that the centers of $\QG$ and $\QG^\Omega$ are the same. Indeed, by definition we have $\Linf(\hh{\QG^\Omega})=\Linf(\hh{\QG})$. Moreover a central subalgebra in this von Neumann algebra is a Baaj-Vaes subalgebra for $\Delta_{\hh{\QG}}$ if and only if it is a Baaj-Vaes subalgebra for $\Delta_{\hh{\QG^\Omega}}$. To see this assume that $\sA$ is a central Baaj-Vaes subalgebra of $\Linf(\hh{\QG})$ for $\Delta_{\hh{\QG}}$. Then $\Delta_{\hh{\QG}}(\sA)\subset\sA\vtens\sA\subset\cZ\bigl(\Linf(\hh{\QG})\bigr)\vtens\cZ\bigl(\Linf(\hh{\QG})\bigr)$ and we get
\[
\bigl.\Delta_{\hh{\QG}}\bigr|_{\sA}=\bigl.\Delta_{\hh{\QG^\Omega}}\bigr|_{\sA}.
\]
Since $\sA$ with $\bigl.\Delta_{\hh{\QG}}\bigr|_{\sA}$ defines a locally compact quantum group $\QH$, using \cite[Proposition 10.5]{BV} and \cite[Proposition 5.45]{KV}, we find that the unitary antipode and scaling group of $\hh{\QG}$ restricted to $\sA$ coincide with the corresponding maps for $\QH$. In particular $\sA$ is a Baaj-Vaes subalgebra in $\Linf(\hh{\QG^\Omega})$ for $\Delta_{\hh{\QG^\Omega}}$.

The equality of centers of $\QG$ and $\QG^\Omega$ follows now from Definition \ref{DefZG}.

\subsection{Duals of Drinfeld-Jimbo deformations}\label{DJ}

We will now consider Drinfeld-Jimbo deformations of a semisimple simply connected compact Lie group (\cite{Drinfeld}). Let $G$ be such a group. As described e.g.~in \cite[Section 2.4]{NeTu} to any value of the deformation parameter $q\in\,]0,1[$ a compact quantum group $\QG_q$ is associated (see also \cite{JosephBook}). We will characterize the center (and consequently also the quantum group of inner automorphisms) of the dual discrete quantum group $\hh{\QG_q}$. In this section we will assume that the deformation parameter $q$ is transcendental. This is needed in order to use results from \cite{JosephBook}.

\begin{theorem}\label{DualGq}
For a Drinfeld-Jimbo quantization $\QG_q$ we have $\cZ(\hh{\QG_q})=\{e\}$. Consequently $\Inn(\hh{\QG_q})=\hh{\QG_q}$.
\end{theorem}

\begin{proof}
Recall from Theorem \ref{LhatGthm} that $\Linf\bigl(\hh{\cZ(\QG)}\bigr)$ coincides with
\[
\hh{\sL}(\hh{\QG_q})={\bigl\{x\in\Linf(\QG_q)\st}\Delta_{\QG_q}(x)\in\Linf(\QG_q)\vtens\cZ(\Linf(\QG_q))\bigr\}.
\]
Moreover $\hh{\sL}(\hh{\QG_q})$ is a left coideal, so we can consider it as a von Neumann algebra with an action of $\QG_q$. Since the von Neumann algebraic version of Podle\'s condition holds automatically (\cite[Corollary 2.9]{proj}) we may decompose $\hh{\sL}(\hh{\QG_q})$ into isotypical components. More precisely, there exists a family $\{\mathcal{V}_i\}_{i\in\mathcal{I}}$ of irreducible $\Pol(\QG_q)$-comodules such that $\mathcal{V}_i\subset\Pol(\QG_q)\subset\hh{\sL}(\hh{\QG_q})$ for each ${i\in\mathcal{I}}$ and as a $\Pol(\QG_q)$-comodule
\[
\hh{\sL}(\hh{\QG_q})=\overline{\bigoplus_i\mathcal{V}_i}^{\textrm{weak}}.
\]

Now \cite[Theorem 9.3.20]{JosephBook} shows that $\cZ\bigl(\Pol(\QG_q)\bigr)=\CC\I$. The facts stated above imply then that for each ${i\in\mathcal{I}}$ we have $\mathcal{V}_i=\CC\I$. The ergodicity of the action $\QG_q$ on $\hh{\sL}(\hh{\QG_q})$ implies now that $\hh{\sL}(\hh{\QG_q})=\CC\I$ and the rest of the theorem follows.
\end{proof}

\begin{remark}
The proof of Theorem \ref{DualGq} will work equally well for any compact quantum group such that $\Pol(\QG)$ has trivial center. Hence it can be restated in the following more general form.
\end{remark}

\begin{theorem}
Let $\QG$ be a compact quantum group with $\cZ\bigl(\Pol(\QG)\bigr)=\CC\I$. Then $\cZ(\hh{\QG})=\{e\}$ and $\Inn(\hh{\QG})=\hh{\QG}$.
\end{theorem}

\subsection*{Acknowledgements}
We thank Uli Kr\"ahmer for providing the key reference for the proof of Theorem \ref{DualGq} and Alexandru Chirvasitu, Sergey Neshveyev and Kenny De Commer for useful remarks on the first version of this paper. The authors would also like to express their gratitude to the referee for suggesting several insightful extensions and generalizations of the original results. The second author was partially supported by National Science Center (NCN) grant no.~2014/14/E/ST1/00525, the first and third authors were supported by National Science Center (NCN) grant no.~2015/17/B/ST1/00085.

\end{document}